\documentclass[12pt,reqno]{amsart}
\usepackage{amsmath,amsthm,amssymb,amsfonts,amscd}
\usepackage{mathrsfs}
\usepackage{bbding}
\usepackage{graphicx,latexsym}
\usepackage{hyperref}

\newcommand{\bea}{\begin{eqnarray}}
\newcommand{\eea}{\end{eqnarray}}
\newcommand{\bna}{\begin{eqnarray*}}
\newcommand{\ena}{\end{eqnarray*}}

\numberwithin{equation}{section} % ¶¨Ò幫ʽ°´½Ú±àºÅ

\theoremstyle{plain}
\newtheorem{theorem}{Theorem}[section]
\newtheorem{lemma}{Lemma}[section]
\newtheorem{corollary}{Corollary}

\theoremstyle{definition}

\begin{document}

\title{A note on simultaneous nonvanishing of Dirichlet $L$-functions and
twists of Hecke-Maass $L$-functions}

\author{Qingfeng Sun}
\address{School of Mathematics and Statistics \\ Shandong University, Weihai \\ Weihai \\Shandong 264209 \\China}
\email{qfsun@sdu.edu.cn}
\date{\today}

\begin{abstract}
 In this note, we prove that given a Hecke-Maass cusp form $f$ for
$SL_2(\mathbb{Z})$ and a sufficiently large integer $q=q_1q_2$
with $q_j\asymp \sqrt{q}$ being prime numbers for $j=1,2$, there exists a primitive
Dirichlet
character $\chi$ of conductor $q$ such that
$L\left(\frac{1}{2},f\otimes \chi\right)
L\left(\frac{1}{2},\chi\right)\neq 0$. To prove this, we establish asymptotic
formulas of $L\left(\frac{1}{2},f\otimes \chi\right)
L\left(\frac{1}{2},\chi\right)$ over the family of even primitive Dirichlet characters $\chi$
of conductor $q$ for more general $q$.
\end{abstract}

\maketitle

\section{Introduction}

The special values of $L$-functions often carry important information,
algebraic, analytic or geometric. Thus it is of the first importance
to see whether it is nonzero. Moreover, in many applications, one is more
concerned with that when two or more $L$-functions are simultaneous
nonvanishing (see \cite{BM}, \cite{IS}, \cite{KMV}, \cite{MV} for example).
Recently, Das and Khan \cite{DK} showed that given a Hecke-Maass cusp
form $f$ for $SL_2(\mathbb{Z})$ and a sufficiently large prime $q$,
there exists a primitive Dirichlet character $\chi$ of conductor $q$ such that
the product of $L$-values $L\left(\frac{1}{2},f\otimes \chi\right)$
and $L\left(\frac{1}{2},\chi\right)$ does not vanish. More precisely, they
proved the following asymptotic formula
\bea
\sideset{}{^\dag}\sum_{\chi \bmod q\atop \chi(-1)=1}
L\left(\frac{1}{2},f \otimes \chi\right)
\overline{L\left(\frac{1}{2},\chi\right)}=
\frac{q-2}{2}L(1,f)
+O_{f,\varepsilon}\left(q^{\frac{7}{8}+\theta+\varepsilon}\right),
\eea
where throughout the paper, the $\dag$ means that the summation is over primitive
characters and $\theta$ denotes the exponent
towards the Ramanujan-Petersson conjecture for $f$, which can be taken
as $\theta=\frac{7}{64}$ due to Kim and Sarnak \cite{K}. They also note that
their method may work for any large integer $q$.
So the aim of this note is to generalize their result to large integer $q=q_1q_1$,
where $q_1$ and $q_2$ are primes satisfying some conditions.
Our main result is the following theorem.

\begin{theorem}
Let $f$ be a Hecke-Maass cusp form for $SL_2(\mathbb{Z})$ and let $q=q_1q_2$,
$q_1$ and $q_2$ being primes.
For any $\varepsilon>0$, we have
\bna
\sideset{}{^\dag}\sum_{\chi \bmod q\atop \chi(-1)=1}
L\left(\frac{1}{2},f \otimes \chi\right)
\overline{L\left(\frac{1}{2},\chi\right)}&=&
\frac{\phi(q)}{2}\left(1-\frac{\lambda_f(q_1)}{q_1}+\frac{1}{q_1^2}\right)
\left(1-\frac{\lambda_f(q_2)}{q_2}+\frac{1}{q_2^2}\right)L(1,f)\\
&+&O\left(q^{\frac{7}{8}+\frac{3\theta}{8(1+\theta)}+\varepsilon}
+q^{\frac{3}{4}+\frac{3\theta}{2}+\varepsilon}
+\left(\frac{q_1}{q_2}+\frac{q_2}{q_1}\right)q^{\frac{3}{4}+\varepsilon}\right)\\
&+&O\left(\frac{q^{\frac{5}{4}+\frac{3\theta}{2}+\varepsilon}}{\min\{q_1,q_2\}}
+\max\{q_1,q_2\}+\max\{q_1,q_2\}^3q^{-\frac{9(1+2\theta)}{8(1+\theta)}+\varepsilon}\right),
\ena
where the implied constants depend on $f$ and $\varepsilon$.
\end{theorem}

\begin{corollary}
Let $f$ be a Hecke-Maass cusp form for $SL_2(\mathbb{Z})$. Let $q=q_1q_2$, $q_1$, $q_2$
being primes and $q^{\frac{1}{2}-\varepsilon}\ll q_1 \ll q^{\frac{1}{2}+\varepsilon}$ with
$\varepsilon>0$ an arbitrarily small constant. Then we have
\bna
\sideset{}{^\dag}\sum_{\chi \bmod q\atop \chi(-1)=1}
L\left(\frac{1}{2},f \otimes \chi\right)
\overline{L\left(\frac{1}{2},\chi\right)}&=&
\frac{\phi(q)}{2}\left(1-\frac{\lambda_f(q_1)}{q_1}+\frac{1}{q_1^2}\right)
\left(1-\frac{\lambda_f(q_2)}{q_2}+\frac{1}{q_2^2}\right)L(1,f)\\
&+&O\left(q^{\frac{7}{8}+\frac{3\theta}{8(1+\theta)}+\varepsilon}
+q^{\frac{3}{4}+\frac{3\theta}{2}+\varepsilon}\right),
\ena
where the implied constant
depends on $f$ and $\varepsilon$.
\end{corollary}

Recently, Liu \cite{L} proved an asymptotic formula for $L\left(\frac{1}{2},f \otimes \chi\right)
\overline{L\left(\frac{1}{2},\chi\right)}$ over the family of even primitive
Dirichlet characters of conductors satisfying a special structure. More precisely,
the Dirichlet character $\chi$ modulo $q$ in Liu's case satisfies the conditions:
$$
q=q_1q_2, q\asymp Q, q_1\asymp Q^{3/4}, q_2\asymp Q^{1/4}.
$$
So the modulus of the Dirichlet character is not specified and our $q$ in Corollary
1.1 is not included in Liu's modulus set. Since $L(1,f) \neq 0$, we also have the following
non-vanishing result.

\begin{corollary}
Let $f$ be a fixed Hecke-Maass cusp form for $SL_2(\mathbb{Z})$. Then for every large
integer $q=q_1q_2$, $q_1$, $q_2$
being primes and $q^{\frac{1}{2}-\varepsilon}\ll q_1 \ll q^{\frac{1}{2}+\varepsilon}$ with
$\varepsilon>0$ an arbitrarily small constant, there exists a primitive Dirichlet character
$\chi$ of conductor $q$ such that the product of the central values
$L\left(\frac{1}{2},f\otimes \chi\right)$
and $L\left(\frac{1}{2},\chi\right)$ does not vanish.
\end{corollary}

The term $q^{\frac{3}{4}+\frac{3\theta}{2}+\varepsilon}$ in Corollary 1.1 can be removed using an
unbalanced approximate functional equation in the proof. This can be seen more explicitly
in the case that $q$ is a prime. In fact, we can prove the following asymptotic formula.

\begin{theorem}
Let $f$ be a Hecke-Maass cusp form for $SL_2(\mathbb{Z})$ and let $q$ be a prime number.
For any $\varepsilon>0$, we have
\bna
\sideset{}{^\dag}\sum_{\chi \bmod q \atop \chi(-1)=1}
L\left(\frac{1}{2},f \otimes \chi\right)
\overline{L\left(\frac{1}{2},\chi\right)}=
\frac{q-2}{2}L(1,f)
+O\left(q^{\frac{7}{8}+\frac{3\theta}{8(1+\theta)}+\varepsilon}\right),
\ena
where the implied constant
depends on $f$ and $\varepsilon$.
\end{theorem}

Theorem 1.2 makes a slight improvement of (1.1). Its proof is similar as that of Theorem 1.1
and much easier. So we omit the details here.

\medskip

\noindent {\bf Notation.} In this paper, $\varepsilon$ is an arbitrarily small positive constant
which is not necessarily the same at each occurrence. Also, the implied constants
in $\ll$ and $O$ depend on $f$ and $\varepsilon$ throughout the paper.

\section{Preliminaries}
\setcounter{equation}{0}
\bigskip

Let $\chi$ be an even primitive Dirichlet character modulo $q$. For Re$(s)>1$
we define the Dirichlet $L$-function
\bna
L(s,\chi)=\sum_{n\geq 1}\chi(n)n^{-s}
\ena
which has analytic continuation to all $s\in \mathbb{C}$ and satisfies
a functional equation relating $s$ and $1-s$.

Let $f$ be a Hecke-Maass cusp form for $SL_2(\mathbb{Z})$ with Laplace eigenvalue
$\frac{1}{4}+t^2$, $t\in \mathbb{R}$. Let $\lambda_f(n)$ be the normalized
$n$-th Fourier coefficient of $f$. For Re$(s)>1$
we define the Hecke $L$-function
\bna
L(s,f\otimes\chi)=\sum_{n\geq 1}\lambda_f(n)\chi(n)n^{-s}
\ena
which also has analytic continuation to the whole complex plane and satisfies
a functional equation relating $s$ and $1-s$. For $L(s,\chi)$ and $L(s,f\otimes\chi)$,
we have the following approximate functional equation (see \cite{IK}, Theorem 5.3).

\begin{lemma}
Let $G(u)=e^{u^2}$. For $\chi$ an even primitive Dirichlet character of modulus $q$, we have
\bna
L\left(\frac{1}{2},f\otimes\chi\right)
\overline{L\left(\frac{1}{2},\chi\right)}&=&
\sum_{m\geq 1}\sum_{n\geq 1}\frac{\lambda_f(n)\chi(n)\overline{\chi(m)}}{\sqrt{m n}}
V\left(\frac{\pi^{\frac{3}{2}}m n}{q^{\frac{3}{2}}}\right)\\
&&+
\frac{\tau(\chi)}{\sqrt{q}}\sum_{m\geq 1}\sum_{n\geq 1}
\frac{\lambda_f(n)\overline{\chi}(n)\chi(m)}{\sqrt{mn}}
V\left(\frac{\pi^{\frac{3}{2}}mn}{q^{\frac{3}{2}}}\right),
\ena
where
\bea
V(y)=\frac{1}{2\pi i}\int\limits_{(1)}y^{-u}
\frac{\Gamma\left(\frac{1+2u+2it}{4}\right)
\Gamma\left(\frac{1+2u-2it}{4}\right)
\Gamma\left(\frac{1+2u}{4}\right)}
{\Gamma\left(\frac{1+2it}{4}\right)
\Gamma\left(\frac{1-2it}{4}\right)
\Gamma\left(\frac{1}{4}\right)}G(u)\frac{\mathrm{d}u}{u}.
\eea
\end{lemma}

The function $V(y)$ has the following properties.
\begin{lemma}
(i) \,For any $y>0$, we have
\bna
V(y)=1+O_{f,\varepsilon}\left(y^{\frac{1}{2}-\varepsilon}\right)
\ena
for any $\varepsilon>0$, and
\bna
y^jV^{(j)}(y)\ll_{f,A,j} (1+y)^{-A}
\ena
for any $A>0$ and $j\geq 0$.
\end{lemma}

We need the following uniform estimate of Fourier coefficients in exponential sums
(see \cite{I}, Theorem 8.1).

\begin{lemma}
For any $\alpha \in \mathbb{R}$, we have
\bna
\sum_{n\leq N}\lambda_f(n)e(\alpha n)\ll_{f,\varepsilon}N^{\frac{1}{2}+\varepsilon},
\ena
uniformly in $\alpha$.
\end{lemma}

The following Voronoi formula can be found in \cite{MS} (see also \cite{G}, Theorem 3.2).

\begin{lemma}
Let $\psi$ be a fixed smooth function with compact support on $\mathbb{R}^+$.
Let $d$, $\overline{d}\in \mathbb{Z}$ with $(c,d)=1$ and $d\overline{d}\equiv 1 (\bmod c)$.
Then
\bna
\sum_{n\geq 1}\lambda_f(n)e\left(\frac{n\overline{d}}{c}\right)
\psi\left(\frac{n}{N}\right)=c\sum_{\pm}\sum_{n\geq 1}\frac{\lambda_f(n)}{n}
e\left(\pm\frac{nd}{c}\right)\Psi^{\pm}\left(\frac{nN}{c^2}\right),
\ena
where for $\sigma>-1$,
\bna
\Psi^{\pm}(x)=\frac{1}{2\pi i}\int\limits_{(\sigma)}(\pi^2x)^{-s}
G^{\pm}(s)\widetilde{\psi}(-s)\mathrm{d}s.
\ena
Here $\widetilde{\psi}(s)=\int_0^{\infty}\psi(x)x^{s-1}\mathrm{d}x$ is the Mellin transform of $\psi(x)$ and
\bna
2\pi G^{\pm}(s)=\frac{\Gamma\left(\frac{1+s+it}{2}\right)\Gamma\left(\frac{1+s-it}{2}\right)}
{\Gamma\left(\frac{-s+it}{2}\right)\Gamma\left(\frac{-s-it}{2}\right)}\pm
\frac{\Gamma\left(\frac{1+s+it+1}{2}\right)\Gamma\left(\frac{1+s-it+1}{2}\right)}
{\Gamma\left(\frac{-s+it+1}{2}\right)\Gamma\left(\frac{-s-it+1}{2}\right)}.
\ena
\end{lemma}
Notice that by shifting the contour of integration to Re$(s)=A$ to any $A>0$, we have
$\Psi^{\pm}(x)\ll_{f,A} x^{-A}$. For small $x$, we move the contour of integration to
Re$(s)=-1+\varepsilon$ to get
\bea
\Psi^{\pm}(x)\ll_{f,\varepsilon} x^{1-\varepsilon}
\eea
for any $\varepsilon>0$.
\medskip

\section{Proof of Theorem 1.1}
\setcounter{equation}{0}
\medskip

Applying the approximate functional equation in Lemma 2.1, we have
\bea
\sideset{}{^\dag}\sum_{\chi \bmod q \atop \chi(-1)=1}
L\left(\frac{1}{2},f \otimes \chi\right)
\overline{L\left(\frac{1}{2},\chi\right)}=S_1+S_2,
\eea
where
\bna
S_1&=&\sum_{m\geq 1}\frac{1}{\sqrt{m}}\sum_{n\geq 1}
\frac{\lambda_f(n)}{\sqrt{n}}
V\left(\frac{\pi^{\frac{3}{2}}m n}{q^{\frac{3}{2}}}\right)
\sideset{}{^\dag}\sum_{\chi \bmod q \atop \chi(-1)=1}
\overline{\chi}(m)\chi(n),\\
S_2&=&\frac{1}{\sqrt{q}}\sum_{m\geq 1}\frac{1}{\sqrt{m}}
\sum_{n\geq 1}\frac{\lambda_f(n)}{\sqrt{n}}
V\left(\frac{\pi^{\frac{3}{2}}m n}{q^{\frac{3}{2}}}\right)
\sideset{}{^\dag}\sum_{\chi \bmod q \atop \chi(-1)=1}
\chi(m)\overline{\chi}(n)\tau(\chi).
\ena
We will show in the next two sections that
\bea
S_1&=&\frac{\phi(q)}{2}\left(1-\frac{\lambda_f(q_1)}{q_1}+\frac{1}{q_1^2}\right)
\left(1-\frac{\lambda_f(q_2)}{q_2}+\frac{1}{q_2^2}\right)L(1,f)\nonumber\\
&&+O_{f,\varepsilon}\left(q_1+q_2+\left(\frac{q_1}{q_2}+\frac{q_2}{q_1}\right)
q^{\frac{3}{4}+\varepsilon}+
q^{\frac{3}{4}+\frac{3\theta}{2}+\varepsilon}\right)
\eea
and
\bea
S_2\ll_{f,\varepsilon}q^{\frac{7}{8}+\frac{3\theta}{8(1+\theta)}+\varepsilon}
+\left(\frac{q_1}{q_2}+\frac{q_2}{q_1}\right)q^{\frac{1}{4}+\frac{3\theta}{2}+\varepsilon}
+\frac{q^{\frac{5}{4}+\frac{3\theta}{2}+\varepsilon}}{\min\{q_1,q_2\}}
+\max\{q_1,q_2\}^3q^{-\frac{9(1+2\theta)}{8(1+\theta)}+\varepsilon}.
\eea
Then Theorem 1.1 follows from (3.1)-(3.3).

\medskip

\section{Estimation of $S_1$}
\setcounter{equation}{0}
\medskip

In this section we prove (3.2). We write
\bea
S_1&=&\sum_{m\geq 1}\frac{1}{\sqrt{m}}\sum_{n\geq 1}
\frac{\lambda_f(n)}{\sqrt{n}}
V\left(\frac{\pi^{\frac{3}{2}}m n}{q^{\frac{3}{2}}}\right)
B_q(m,n)
\eea
where
$
B_q(m,n)=\sideset{}{^\dag}\sum\limits_{\chi \bmod q \atop \chi(-1)=1}
\overline{\chi}(m)\chi(n).
$
By the orthogonality of Dirichlet characters, we have for $(mn,q)=1$,
\bea
B_q(m,n)&=&\frac{1}{2}\quad\sideset{}{^\dag}\sum_{\chi \bmod q}
(1+\chi(-1))\overline{\chi}(m)\chi(n)\nonumber\\
&=&\frac{1}{2}\sum_{\pm}\quad\sideset{}{^\dag}\sum_{\chi \bmod q}
\overline{\chi}(m)\chi(\pm n)\nonumber\\
&=&\frac{1}{2}\sum_{\pm}\quad\sideset{}{^\dag}\sum_{\chi_1 \bmod q_1}
\overline{\chi_1}(m)\chi_1(\pm n)
\quad\sideset{}{^\dag}\sum_{\chi_2 \bmod q_2}
\overline{\chi_2}(m)\chi_2(\pm n)\nonumber\\
&=&\frac{1}{2}\sum_{\pm}\left[\phi(q_1)1_{n\equiv \pm m\bmod q_1}-1\right]
\left[\phi(q_2)1_{n\equiv \pm m\bmod q_2}-1\right].
\eea
Plugging (4.2) into (4.1) we have
\bea
S_1=S_{11}+S_{12}-S_{13}-S_{14},
\eea
where
\bna
S_{11}&=&\frac{\phi(q)}{2}\sum_{\pm}\sum_{m\geq 1 \atop (m,q)=1}
\frac{1}{\sqrt{m}}\sum_{n\geq 1 \atop n\equiv \pm m \bmod q}
\frac{\lambda_f(n)}{\sqrt{n}}
V\left(\frac{\pi^{\frac{3}{2}}m n}{q^{\frac{3}{2}}}\right),\\
S_{12}&=&\sum_{m\geq 1 \atop (m,q)=1}
\frac{1}{\sqrt{m}}\sum_{n\geq 1 \atop (n,q)=1}
\frac{\lambda_f(n)}{\sqrt{n}}
V\left(\frac{\pi^{\frac{3}{2}}m n}{q^{\frac{3}{2}}}\right),\\
S_{13}&=&\frac{\phi(q_1)}{2}\sum_{\pm}\sum_{m\geq 1 \atop (m,q)=1}
\frac{1}{\sqrt{m}}\sum_{n\geq 1, (n,q_2)=1 \atop n\equiv \pm m \bmod q_1}
\frac{\lambda_f(n)}{\sqrt{n}}
V\left(\frac{\pi^{\frac{3}{2}}m n}{q^{\frac{3}{2}}}\right),\\
S_{14}&=&\frac{\phi(q_2)}{2}\sum_{\pm}\sum_{m\geq 1 \atop (m,q)=1}
\frac{1}{\sqrt{m}}\sum_{n\geq 1, (n,q_1)=1 \atop n\equiv \pm m \bmod q_2}
\frac{\lambda_f(n)}{\sqrt{n}}
V\left(\frac{\pi^{\frac{3}{2}}m n}{q^{\frac{3}{2}}}\right).
\ena
Trivially, we have
\bea
S_{12}\ll_{f,\varepsilon} \sum_{m\ll q^{\frac{3}{2}+\varepsilon}}\frac{1}{\sqrt{m}}
\sum_{n\ll \frac{q^{\frac{3}{2}+\varepsilon}}{m}}
\frac{|\lambda_f(n)|}{\sqrt{n}}+1
\ll q^{\frac{3}{4}+\varepsilon}.
\eea

Next, we show that $S_{11}$ contributes the main term.
\begin{lemma}
For any $\varepsilon>0$ we have
\bna
S_{11}&=&\frac{\phi(q)}{2}\left(1-\frac{\lambda_f(q_1)}{q_1}
+\frac{1}{q_1^2}\right)
\left(1-\frac{\lambda_f(q_2)}{q_2}+\frac{1}{q_2^2}\right)L(1,f)\nonumber\\
&&+O_{f,\varepsilon}\left(q^{\frac{3}{4}+\frac{3\theta}{2}+\varepsilon}
+\left(\frac{q_1}{q_2}+\frac{q_2}{q_1}\right) q^{\frac{3}{4}+\varepsilon}\right).
\ena
\end{lemma}

\begin{proof}
We write
\bea
S_{11}=S_{11}^0+S_{11}^{\flat},
\eea
where $S_{11}^0$ is the diagonal term from $n=\pm m$ and $S_{11}^{\flat}$ is the remaining
terms. Then
\bna
S_{11}^0=\frac{\phi(q)}{2}\sum_{m\geq 1 \atop (m,q)=1}
\frac{\lambda_f(m)}{m}
V\left(\frac{\pi^{\frac{3}{2}}m^2}{q^{\frac{3}{2}}}\right).
\ena
By the definition of $V(y)$ in (2.1), we have
\bna
S_{11}^0
&=&\frac{\phi(q)}{2}\frac{1}{2\pi i}\int\limits_{(1)}
\left\{\sum_{m\geq 1 \atop (m,q)=1}\frac{\lambda_f(m)}{m^{1+2u}}\right\}
\left(\frac{\pi^{\frac{3}{2}}}{q^{\frac{3}{2}}}\right)^{-u}
\frac{\Gamma\left(\frac{1+2u+2it}{4}\right)
\Gamma\left(\frac{1+2u-2it}{4}\right)
\Gamma\left(\frac{1+2u}{4}\right)}
{\Gamma\left(\frac{1+2it}{4}\right)
\Gamma\left(\frac{1-2it}{4}\right)
\Gamma\left(\frac{1}{4}\right)}G(u)\frac{\mathrm{d}u}{u}\\
&=&\frac{\phi(q)}{2}\frac{1}{2\pi i}\int\limits_{(1)}
\prod_{p|q}\left(1-\frac{\lambda_f(p)}{p^{1+2u}}+\frac{1}{p^{2+4u}}\right)L(1+2u,f)
\left(\frac{\pi^{\frac{3}{2}}}{q^{\frac{3}{2}}}\right)^{-u}\\
&&\frac{\Gamma\left(\frac{1+2u+2it}{4}\right)
\Gamma\left(\frac{1+2u-2it}{4}\right)
\Gamma\left(\frac{1+2u}{4}\right)}
{\Gamma\left(\frac{1+2it}{4}\right)
\Gamma\left(\frac{1-2it}{4}\right)
\Gamma\left(\frac{1}{4}\right)}G(u)\frac{\mathrm{d}u}{u}.
\ena
Shifting the contour of integration to $\mbox{Re}(u)=-\frac{1}{2}+\varepsilon$, we have
\bea
S_{11}^0=\frac{\phi(q)}{2}\left(1-\frac{\lambda_f(q_1)}{q_1}+\frac{1}{q_1^2}\right)
\left(1-\frac{\lambda_f(q_2)}{q_2}+\frac{1}{q_2^2}\right)L(1,f)
+O\left(q^{\frac{1}{4}+\varepsilon}\right).
\eea

Next we bound $S_{11}^{\flat}$. By applying smooth partitions of unity to the variables
$m$ and $n$, we are led to estimating
\bna
S_{11}^{\flat}(\pm, M_1, N_1)=\phi(q)\sum_{(m,q)=1}
\frac{1}{\sqrt{m}}\omega_1\left(\frac{m}{M_1}\right)
\sum_{n\equiv \pm m \bmod q \atop n\neq \pm m }
\frac{\lambda_f(n)}{\sqrt{n}}\omega_2\left(\frac{n}{N_1}\right)
V\left(\frac{\pi^{\frac{3}{2}}m n }{q^{\frac{3}{2}}}\right),
\ena
with $M_1N_1\ll q^{\frac{3}{2}+\varepsilon}$ by the properties of $V(y)$ in Lemma 2.2,
where $\omega_j(x)$, $j=1,2$, are smooth functions on $[1,2]$ satisfying
$\omega_j(x)\ll_j 1$.

For $M_1<q/3$, we have, trivially,
\bea
S_{11}^{\flat}(\pm, M_1, N_1)\ll q N_1^{\theta}\sum_{m}
\frac{1}{\sqrt{m}}\omega_1\left(\frac{m}{M_1}\right)
\sum_{1\leq k<N_1/q}
\frac{1}{\sqrt{kq}}
\ll N_1^{\theta} \sqrt{M_1N_1}
\ll q^{\frac{3}{4}+\frac{3\theta}{2}+\varepsilon}.
\eea
For $M_1\geq q/3$, the condition $m\neq \pm n$ is moot. Using the relation
\bna
\frac{1}{q}\sum_{r|q}\sideset{}{^*}\sum_{a\bmod r}
e\left(\frac{(n\mp m)a}{r}\right)
=\left\{\begin{array}{ll}
1, &\mbox{if $n \equiv\pm m \bmod q$},\\
0, &\mbox{otherwise},
\end{array}
\right.
\ena
where the $*$ denotes that the summation is restricted by the condition
$(a,r)=1$,
we have (notice that for $q=q_1q_2$, we have $r=1,q_1,q_2$ or $q$)
\bea
&&S_{11}^{\flat}(\pm, M_1, N_1)\nonumber\\
&=&\frac{\phi(q)}{q}\sum_{(m,q)=1}
\frac{1}{\sqrt{m}}\omega_1\left(\frac{m}{M_1}\right)
\sum_{n\geq 1}
\frac{\lambda_f(n)}{\sqrt{n}}\omega_2\left(\frac{n}{N_1}\right)
V\left(\frac{\pi^{\frac{3}{2}}m n}{q^{\frac{3}{2}}}\right)\nonumber\\
&&+\frac{\phi(q)}{q}\sum_{(m,q)=1}
\frac{1}{\sqrt{m}}\omega_1\left(\frac{m}{M_1}\right)
\sum_{n\geq 1}
\frac{\lambda_f(n)}{\sqrt{n}}\omega_2\left(\frac{n}{N_1}\right)
V\left(\frac{\pi^{\frac{3}{2}}m n}{q^{\frac{3}{2}}}\right)
\sideset{}{^*}\sum_{a\bmod q_1}
e\left(\frac{(n\mp m)a}{q_1}\right)\nonumber\\
&&+\frac{\phi(q)}{q}\sum_{(m,q)=1}
\frac{1}{\sqrt{m}}\omega_1\left(\frac{m}{M_1}\right)
\sum_{n\geq 1}
\frac{\lambda_f(n)}{\sqrt{n}}\omega_2\left(\frac{n}{N_1}\right)
V\left(\frac{\pi^{\frac{3}{2}}m n}{q^{\frac{3}{2}}}\right)
\sideset{}{^*}\sum_{a\bmod q_2}
e\left(\frac{(n\mp m)a}{q_2}\right)\nonumber\\
&&+\frac{\phi(q)}{q}\sum_{(m,q)=1}
\frac{1}{\sqrt{m}}\omega_1\left(\frac{m}{M_1}\right)
\sum_{n\geq 1}
\frac{\lambda_f(n)}{\sqrt{n}}\omega_2\left(\frac{n}{N_1}\right)
V\left(\frac{\pi^{\frac{3}{2}}m n}{q^{\frac{3}{2}}}\right)
\sideset{}{^*}\sum_{a\bmod q}
e\left(\frac{(n\mp m)a}{q}\right)\nonumber\\
&=&\sum_{j=1}^4E_j(\pm,M_1,N_1),
\eea
say.

Trivially, we have
\bea
E_1(\pm,M_1,N_1)\ll \sqrt{M_1 N_1}\ll q^{\frac{3}{4}+\varepsilon}.
\eea

To estimate $E_2(\pm,M_1,N_1)$, we note that $(m,q)\neq 1$ implies that
$q_1|m$ or $q_2|m$, whereas
\bna
&&\frac{\phi(q)}{q}\sum_{q_i|m}
\frac{1}{\sqrt{m}}\omega_1\left(\frac{m}{M_1}\right)
\sum_{n\geq 1}
\frac{\lambda_f(n)}{\sqrt{n}}\omega_2\left(\frac{n}{N_1}\right)
V\left(\frac{\pi^{\frac{3}{2}}m n}{q^{\frac{3}{2}}}\right)
\sideset{}{^*}\sum_{a\bmod q_1}
e\left(\frac{(n\mp m)a}{q_1}\right)\nonumber\\
&\ll& q_1\sum_{m\geq 1}
\frac{1}{\sqrt{mq_i}}\omega_1\left(\frac{mq_i}{M_1}\right)
\sum_{n\geq 1}
\frac{|\lambda_f(n)|}{\sqrt{n}}\omega_2\left(\frac{n}{N_1}\right)
V\left(\frac{\pi^{\frac{3}{2}}mq_i n}{q^{\frac{3}{2}}}\right)\nonumber\\
&\ll& \frac{q_1}{q_i}q^{\frac{3}{4}+\varepsilon}.
\ena
Thus
\bna
E_2(\pm,M_1,N_1)&=&\frac{\phi(q)}{q}\sum_{n\geq 1}
\frac{\lambda_f(n)}{\sqrt{n}}\omega_2\left(\frac{n}{N_1}\right)\sideset{}{^*}\sum_{a\bmod q_1}
e\left(\frac{na}{q_1}\right)\sum_{m\geq 1}
\frac{1}{\sqrt{m}}\omega_1\left(\frac{m}{M_1}\right)\nonumber\\
&&V\left(\frac{\pi^{\frac{3}{2}}m n}{q^{\frac{3}{2}}}\right)
e\left(\frac{\mp ma}{q_1}\right)+O\left(\frac{q_1}{q_2}q^{\frac{3}{4}+\varepsilon}
+q^{\frac{3}{4}+\varepsilon}\right).
\ena
Reducing the $m$-sum to the residue classes and applying
Poisson summation formula, we have
\bna
m\mbox{-sum}&=&\sum_{\gamma\bmod q_1}
e\left(\frac{\mp \gamma a}{q_1}\right)\sum_{m \equiv \gamma \bmod q_1}
\frac{1}{\sqrt{m}}\omega_1\left(\frac{m}{M_1}\right)
V\left(\frac{\pi^{\frac{3}{2}}m n}{q^{\frac{3}{2}}}\right)\\
&=&\frac{1}{q_1}\sum_{\gamma\bmod q_1}
e\left(\frac{\mp \gamma a}{q_1}\right)\sum_{m\in \mathbb{Z}}
e\left(\frac{\gamma m}{q_1}\right)
\int_{\mathbb{R}}\frac{1}{\sqrt{u}}
\omega_1\left(\frac{u}{M_1}\right)
V\left(\frac{\pi^{\frac{3}{2}}n u}{q^{\frac{3}{2}}}\right)
e\left(-\frac{m u}{q_1}\right)\mathrm{d}u\\
&=&\sqrt{M_1}\sum_{m\equiv \pm a\bmod q_1}J(m,n),
\ena
where
\bna
J(m,n)=\int_{\mathbb{R}}\frac{\omega_1(u)}{\sqrt{u}}
V\left(\frac{\pi^{\frac{3}{2}}n M_1 u}{q^{\frac{3}{2}}}\right)
e\left(-\frac{m M_1 u}{q_1}\right)\mathrm{d}u.
\ena
By partial integration $j$ times, we have
\bna
J(m,n)\ll_j\left(1+\frac{|m|M_1}{q_1}\right)^{-j}
\ena
for any $j\geq 0$. Thus
\bna
E_2(\pm,M_1,N_1)
=\frac{\phi(q)\sqrt{M_1}}{q}\sum_{|m|<\frac{q_1^{1+\varepsilon}}{M_1}}\sum_{n\geq 1}
\frac{\lambda_f(n)}{\sqrt{n}}\omega_2\left(\frac{n}{N_1}\right)
J(m,n)e\left(\frac{\pm m n}{q_1}\right)
+O\left(\frac{q_1}{q_2}q^{\frac{3}{4}+\varepsilon}
+q^{\frac{3}{4}+\varepsilon}\right).
\ena
Now we consider the $n$-sum. By partial integration once and Lemma 2.3, we have
\bna
&&\sum_{n\geq 1}
\frac{\lambda_f(n)}{\sqrt{n}}\omega_2\left(\frac{n}{N_1}\right)
J(m,n)e\left(\frac{\pm m n}{q_1}\right)\\
&=&\int_0^{\infty}\frac{1}{\sqrt{u}}\omega_2\left(\frac{u}{N_1}\right)
J(m,u)\mathrm{d}\sum_{n\leq u}\lambda_f(n)e\left(\frac{\pm m n}{q_1}\right)\\
&=&\int_0^{\infty}\left(\sum_{n\leq u}
\lambda_f(n)e\left(\frac{\pm m n}{q_1}\right)\right)
\left(\frac{1}{\sqrt{u}}\omega_2\left(\frac{u}{N_1}\right)
J(m,u)\right)'\mathrm{d}u\\
&\ll& \int_{N_1}^{2N_1} u^{\frac{1}{2}+\varepsilon} u^{-\frac{3}{2}}\mathrm{d}u\\
&\ll& N_1^{\varepsilon}.
\ena
Hence,
\bea
E_2(\pm,M_1,N_1)\ll \frac{q_1^{1+\varepsilon}}{\sqrt{M_1}}+
\frac{q_1}{q_2}q^{\frac{3}{4}+\varepsilon}
+q^{\frac{3}{4}+\varepsilon}
\ll \frac{q_1}{q_2}q^{\frac{3}{4}+\varepsilon}
+q^{\frac{3}{4}+\varepsilon}.
\eea
Similarly, we have
\bea
E_3(\pm,M_1,N_1)
\ll\frac{q_2}{q_1}q^{\frac{3}{4}+\varepsilon}
+q^{\frac{3}{4}+\varepsilon}.
\eea
In the following we bound
\bna
E_4(\pm,M_1,N_1)=\frac{\phi(q)}{q}\sum_{(m,q)=1}
\frac{1}{\sqrt{m}}\omega_1\left(\frac{m}{M_1}\right)
\sum_{n\geq 1}
\frac{\lambda_f(n)}{\sqrt{n}}\omega_2\left(\frac{n}{N_1}\right)
V\left(\frac{\pi^{\frac{3}{2}}m n}{q^{\frac{3}{2}}}\right)
\sideset{}{^*}\sum_{a\bmod q}
e\left(\frac{(n\mp m)a}{q}\right).
\ena
The contribution from $(m,q)\neq 1$ is
\bna
&&\frac{\phi(q)}{q}\sum_{(m,q)\neq 1}
\frac{1}{\sqrt{m}}\omega_1\left(\frac{m}{M_1}\right)
\sum_{n\geq 1}
\frac{\lambda_f(n)}{\sqrt{n}}\omega_2\left(\frac{n}{N_1}\right)
V\left(\frac{\pi^{\frac{3}{2}}m n}{q^{\frac{3}{2}}}\right)
\sideset{}{^*}\sum_{a\bmod q}
e\left(\frac{(n\mp m)a}{q}\right)\\
&=&\Delta_1+\Delta_2+\Delta_3,
\ena
where
\bna
\Delta_1&=&\frac{\phi(q)}{q}\sum_{q|m}
\frac{1}{\sqrt{m}}\omega_1\left(\frac{m}{M_1}\right)
\sum_{n\geq 1}
\frac{\lambda_f(n)}{\sqrt{n}}\omega_2\left(\frac{n}{N_1}\right)
V\left(\frac{\pi^{\frac{3}{2}}m n}{q^{\frac{3}{2}}}\right)
\sideset{}{^*}\sum_{a\bmod q}
e\left(\frac{(n\mp m)a}{q}\right)\\
\Delta_2&=&\frac{\phi(q)}{q}\sum_{q_1|m \atop (m,q_2)=1}
\frac{1}{\sqrt{m}}\omega_1\left(\frac{m}{M_1}\right)
\sum_{n\geq 1}
\frac{\lambda_f(n)}{\sqrt{n}}\omega_2\left(\frac{n}{N_1}\right)
V\left(\frac{\pi^{\frac{3}{2}}m n}{q^{\frac{3}{2}}}\right)
\sideset{}{^*}\sum_{a\bmod q}
e\left(\frac{(n\mp m)a}{q}\right)\\
\Delta_3&=&\frac{\phi(q)}{q}\sum_{q_2|m \atop (m,q_1)=1}
\frac{1}{\sqrt{m}}\omega_1\left(\frac{m}{M_1}\right)
\sum_{n\geq 1}
\frac{\lambda_f(n)}{\sqrt{n}}\omega_2\left(\frac{n}{N_1}\right)
V\left(\frac{\pi^{\frac{3}{2}}m n}{q^{\frac{3}{2}}}\right)
\sideset{}{^*}\sum_{a\bmod q}e\left(\frac{(n\mp m)a}{q}\right).
\ena
We have
\bna
\Delta_1&=&\frac{\phi(q)}{q}\sum_{q|m}
\frac{1}{\sqrt{m}}\omega_1\left(\frac{m}{M_1}\right)
\sum_{n\geq 1}
\frac{\lambda_f(n)}{\sqrt{n}}\omega_2\left(\frac{n}{N_1}\right)
V\left(\frac{\pi^{\frac{3}{2}}m n}{q^{\frac{3}{2}}}\right)
\sideset{}{^*}\sum_{a_1\bmod q_1}e\left(\frac{na_1}{q_1}\right)
\sideset{}{^*}\sum_{a_2\bmod q_2}e\left(\frac{na_2}{q_2}\right)\\
&\ll& \sum_{m}
\frac{1}{\sqrt{mq}}\omega_1\left(\frac{mq}{M_1}\right)
\sum_{n\geq 1}
\frac{|\lambda_f(n)|}{\sqrt{n}}\omega_2\left(\frac{n}{N_1}\right)(n,q_1)(n,q_2)\\
&\ll&\sqrt{\frac{M_1 N_1}{q}}\ll  q^{\frac{1}{4}+\varepsilon}
\ena
and
\bna
\Delta_2&=&\frac{\phi(q)}{q}\sum_{(m,q_2)=1}
\frac{1}{\sqrt{mq_1}}\omega_1\left(\frac{mq_1}{M_1}\right)
\sum_{n\geq 1}
\frac{\lambda_f(n)}{\sqrt{n}}\omega_2\left(\frac{n}{N_1}\right)
V\left(\frac{\pi^{\frac{3}{2}}q_1m n}{q^{\frac{3}{2}}}\right)\\
&&\sideset{}{^*}\sum_{a_1\bmod q_1}e\left(\frac{na_1}{q_1}\right)
\sideset{}{^*}\sum_{a_2\bmod q_2}e\left(\frac{(n\mp mq_1)a_2}{q_2}\right)\\
&=&\frac{\phi(q)}{q}\sum_{(m,q_2)=1}
\frac{1}{\sqrt{mq_1}}\omega_1\left(\frac{mq_1}{M_1}\right)
\sum_{n\geq 1}
\frac{\lambda_f(n)}{\sqrt{n}}\omega_2\left(\frac{n}{N_1}\right)
V\left(\frac{\pi^{\frac{3}{2}}q_1m n}{q^{\frac{3}{2}}}\right)\\
&&\left(q_11_{q_1|n}-1\right)
\left(q_21_{n\equiv \pm mq_1\bmod q_2}-1\right)\\
&\ll& q^{\frac{3}{4}+\frac{3\theta}{2}+\varepsilon},
\ena
where we have used the fact that the condition $n\equiv \pm mq_1 \bmod q_2$
implies that $mq_1\asymp |k|q_2$, $1\leq |k|\ll M_1/q_2$.
Similarly,
\bna
\Delta_3\ll q^{\frac{3}{4}+\frac{3\theta}{2}+\varepsilon}.
\ena
Therefore,
\bna
E_4(\pm,M_1,N_1)
&=&\frac{\phi(q)}{q}
\sum_{n\geq 1}\frac{\lambda_f(n)}{\sqrt{n}}\omega_2\left(\frac{n}{N_1}\right)
\sideset{}{^*}\sum_{a\bmod q}e\left(\frac{na}{q}\right)
\nonumber\\
&&\sum_{m\geq 1}\frac{1}{\sqrt{m}}\omega_1\left(\frac{m}{M_1}\right)
V\left(\frac{\pi^{\frac{3}{2}}m n}{q^{\frac{3}{2}}}\right)
e\left(\frac{\mp ma}{q}\right)
+O\left(q^{\frac{3}{4}+\frac{3\theta}{2}+\varepsilon}\right).
\ena
By Poisson summation formula, we have
\bna
m\mbox{-sum}&=&\sum_{\gamma \bmod q}e\left(\frac{\mp\gamma a}{q}\right)
\sum_{m\equiv \gamma\bmod q}\frac{1}{\sqrt{m}}\omega_1\left(\frac{m}{M_1}\right)
V\left(\frac{\pi^{\frac{3}{2}}m n}{q^{\frac{3}{2}}}\right)\\
&=&\frac{1}{q}\sum_{\gamma \bmod q}e\left(\frac{\mp \gamma a}{q}\right)
\sum_{m\in \mathbb{Z}}e\left(\frac{\gamma m}{q}\right)
\int_{\mathbb{R}}
\frac{1}{\sqrt{u}}\omega_1\left(\frac{u}{M_1}\right)
V\left(\frac{\pi^{\frac{3}{2}}n u}{q^{\frac{3}{2}}}\right)
e\left(-\frac{m u}{q}\right)\mathrm{d}u\\
&=&\frac{\sqrt{M_1}}{q}\sum_{m\in \mathbb{Z}}
\left(\sum_{\gamma \bmod q}e\left(\frac{(m\mp a)\gamma}{q}\right)\right)
\int_{\mathbb{R}}
\frac{\omega_1(u)}{\sqrt{u}}
V\left(\frac{\pi^{\frac{3}{2}}n M u}{q^{\frac{3}{2}}}\right)
e\left(-\frac{m M u}{q}\right)\mathrm{d}u\\
&=&\sqrt{M_1}\sum_{|m|<\frac{q^{1+\varepsilon}}{M_1}
\atop m\equiv \pm a \bmod q}K(m,n),
\ena
where
\bna
K(m,n)=\int_{\mathbb{R}}
\frac{\omega_1(u)}{\sqrt{u}}
V\left(\frac{\pi^{\frac{3}{2}}n M_1u}{q^{\frac{3}{2}}}\right)
e\left(-\frac{m M_1 u}{q}\right)\mathrm{d}u.
\ena
By partial integration $j$ times we have
$K(m,n)\ll_j \left(1+\frac{|m|M}{q}\right)^{-j}$ for any $j\geq 0$.
Thus
\bna
E_4(\pm,M_1,N_1)
&=&\frac{\phi(q)\sqrt{M_1}}{q}\sum_{|m|<\frac{q^{1+\varepsilon}}{M_1}}
\sum_{n\geq 1}\frac{\lambda_f(n)}{\sqrt{n}}\omega_2\left(\frac{n}{N_1}\right)
e\left(\frac{\pm m n}{q}\right)K(m,n)+
O\left(q^{\frac{3}{4}+\frac{3\theta}{2}+\varepsilon}\right).
\ena
As before, by partial integration once and Lemma 2.3, we obtain, for
$M_1\geq q/3$,
\bea
E_4(\pm,M_1,N_1)\ll \frac{q^{1+\varepsilon}}{\sqrt{M_1}}+q^{\frac{3}{4}+\frac{3\theta}{2}+\varepsilon}
\ll q^{\frac{1}{2}+\varepsilon}+q^{\frac{3}{4}+\frac{3\theta}{2}+\varepsilon}.
\eea
By (4.7)-(4.12), we obtain
\bea
S_{11}^{\flat}(\pm, M_1, N_1)\ll q^{\frac{3}{4}+\frac{3\theta}{2}+\varepsilon}
+\left(\frac{q_1}{q_2}+\frac{q_2}{q_1}\right) q^{\frac{3}{4}+\varepsilon}.
\eea
Then Lemma 4.1 follows from (4.5), (4.6) and (4.13).

\end{proof}

We estimate $S_{13}$ and $S_{14}$ similarly to get
\begin{lemma}
For any $\varepsilon>0$, we have
\bna
S_{13}\ll_{f,\varepsilon} q_1+\frac{q_1}{q_2}q^{\frac{3}{4}+\varepsilon}+
q^{\frac{3}{4}+\frac{3\theta}{2}+\varepsilon}
\ena
and
\bna
S_{14}\ll_{f,\varepsilon} q_2+\frac{q_2}{q_1}q^{\frac{3}{4}+\varepsilon}+
q^{\frac{3}{4}+\frac{3\theta}{2}+\varepsilon}.
\ena
\end{lemma}

Now (3.2) follows from (4.3), (4.4) and Lemmas 4.1 and 4.2.

\medskip

\section{Estimation of $S_2$}
\setcounter{equation}{0}
\medskip

Recall that
\bea
S_2&=&\frac{1}{\sqrt{q}}\sum_{m\geq 1 \atop (m,q)=1}\frac{1}{\sqrt{m}}
\sum_{n\geq 1 \atop (n,q)=1}\frac{\lambda_f(n)}{\sqrt{n}}
V\left(\frac{\pi^{\frac{3}{2}}m n}{q^{\frac{3}{2}}}\right)
\mathscr{D}_q(m,n),
\eea
where
\bna
\mathscr{D}_q(m,n)=\sideset{}{^\dag}\sum_{\chi \bmod q \atop \chi(-1)=1}
\chi(m)\overline{\chi}(n)\tau(\chi)
\ena
with $\tau(\chi)=\sum\limits_{a\bmod q}\chi(a)e\left(\frac{a}{q}\right)$.
Note that $\tau(\chi_1\chi_2)=\chi_1(q_2)\chi_2(q_1)\tau(\chi_1)\tau(\chi_2)$.
By the orthogonality of Dirichlet characters, we have
\bea
\mathscr{D}_q(m,n)&=&\frac{1}{2}\sideset{}{^\dag}
\sum_{\chi \bmod q}\left(1+\chi(-1)\right)
\chi(m)\overline{\chi}(n)\tau(\chi)\nonumber\\
&=&\frac{1}{2}\sum_{\pm}\sideset{}{^\dag}
\sum_{\chi \bmod q}\chi(\pm m)\overline{\chi}(n)\tau(\chi)\nonumber\\
&=&\frac{1}{2}\sum_{\pm}
\left(\sideset{}{^\dag}\sum_{\chi_1 \bmod q_1}
\chi_1(\pm mq_2)\overline{\chi_1}(n)\tau(\chi_1)\right)
\left(\sideset{}{^\dag}\sum_{\chi_2 \bmod q_2}
\chi_2(\pm mq_1)\overline{\chi_2}(n)\tau(\chi_2)\right)\nonumber\\
&=&\frac{1}{2}\sum_{\pm}
\left(\sideset{}{^*}\sum_{a \bmod q_1}e\left(\frac{a}{q_1}\right)
\sideset{}{^\dag}\sum_{\chi_1 \bmod q_1}
\chi_1(\pm amq_2)\overline{\chi_1}(n)\right)\nonumber\\
&&\left(\sideset{}{^*}\sum_{b \bmod q_2}e\left(\frac{b}{q_2}\right)
\sideset{}{^\dag}\sum_{\chi_2 \bmod q_2}
\chi_2(\pm bmq_1)\overline{\chi_2}(n)\right)\nonumber\\
&=&\frac{1}{2}\sum_{\pm}\left(\phi(q_1)e\left(\pm\frac{\overline{m}_1
\overline{q_2}n}{q_1}\right)+1\right)
\left(\phi(q_2)e\left(\pm\frac{\overline{m}_2
\overline{q_1}n}{q_2}\right)+1\right),
\eea
where $\overline{m}_1m\equiv 1\bmod q_1$ and
$\overline{m}_2m\equiv 1\bmod q_2$.
Plugging (5.2) into (5.1), we obtain
\bea
S_2&=&\frac{\phi(q)}{2\sqrt{q}}\sum_{\pm}\sum_{m\geq 1 \atop (m,q)=1}\frac{1}{\sqrt{m}}
\sum_{n\geq 1 \atop (n,q)=1}\frac{\lambda_f(n)}{\sqrt{n}}
V\left(\frac{\pi^{\frac{3}{2}}m n}{q^{\frac{3}{2}}}\right)
e\left(\pm\frac{\overline{m}_1
\overline{q_2}n}{q_1}\pm\frac{\overline{m}_2
\overline{q_1}n}{q_2}\right)\nonumber\\
&&+\frac{\phi(q_1)}{2\sqrt{q}}\sum_{\pm}
\sum_{m\geq 1 \atop (m,q)=1}\frac{1}{\sqrt{m}}
\sum_{n\geq 1 \atop (n,q)=1}\frac{\lambda_f(n)}{\sqrt{n}}
V\left(\frac{\pi^{\frac{3}{2}}m n}{q^{\frac{3}{2}}}\right)
e\left(\pm\frac{\overline{m}_1
\overline{q_2}n}{q_1}\right)\nonumber\\
&&+\frac{\phi(q_2)}{2\sqrt{q}}\sum_{\pm}\sum_{m\geq 1 \atop (m,q)=1}\frac{1}{\sqrt{m}}
\sum_{n\geq 1 \atop (n,q)=1}\frac{\lambda_f(n)}{\sqrt{n}}
V\left(\frac{\pi^{\frac{3}{2}}m n}{q^{\frac{3}{2}}}\right)
e\left(\pm\frac{\overline{m}_2\overline{q_1}n}{q_2}\right)\nonumber\\
&&+\frac{1}{\sqrt{q}}\sum_{m\geq 1 \atop (m,q)=1}\frac{1}{\sqrt{m}}
\sum_{n\geq 1 \atop (n,q)=1}\frac{\lambda_f(n)}{\sqrt{n}}
V\left(\frac{\pi^{\frac{3}{2}}m n}{q^{\frac{3}{2}}}\right)\nonumber\\
&=&S_{21}+S_{22}+S_{23}+S_{24},
\eea
say. Trivially, we have
\bea
S_{24}
\ll \frac{1}{\sqrt{q}}\left(q^{\frac{3}{2}+\varepsilon}\right)^{\frac{1}{2}}
\ll q^{\frac{1}{4}+\varepsilon}.
\eea

\begin{lemma}
For any $\varepsilon>0$, we have
\bna
S_{21}\ll_{f,\varepsilon}q^{\frac{7}{8}+\frac{3\theta}{8(1+\theta)}+\varepsilon}
+\frac{q^{\frac{5}{4}+\frac{3\theta}{2}+\varepsilon}}{\min\{q_1,q_2\}}.
\ena
\end{lemma}
\begin{proof}
Applying smooth partitions of unity to the variables $m$ and $n$, respectively,
we need to bound
\bna
S_{21}(\pm, M_2, N_2)&=&\frac{\phi(q)}{\sqrt{q}}
\sum_{(m,q)=1}\frac{1}{\sqrt{m}}
\omega_1\left(\frac{m}{M_2}\right)
\sum_{(n,q)=1}\frac{\lambda_f(n)}{\sqrt{n}}
\omega_2\left(\frac{n}{N_2}\right)\\
&&V\left(\frac{\pi^{\frac{3}{2}}m n}{q^{\frac{3}{2}}}\right)
e\left(\pm\frac{\overline{m}_1
\overline{q_2}n}{q_1}\pm\frac{\overline{m}_2
\overline{q_1}n}{q_2}\right)
\ena
with $M_2N_2\ll q^{\frac{3}{2}+\varepsilon}$ for any $\varepsilon>0$.
Removing the condition $(n,q)=1$ at a cost of
$O\left(\frac{q^{\frac{5}{4}+\frac{3\theta}{2}+\varepsilon}}{\min\{q_1,q_2\}}
\right)$, we have
\bea
S_{21}(\pm, M_2, N_2)&=&\frac{\phi(q)}{\sqrt{q}}
\sum_{(m,q)=1}\frac{1}{\sqrt{m}}
\omega_1\left(\frac{m}{M_2}\right)
\sum_{n\geq 1}\frac{\lambda_f(n)}{\sqrt{n}}
\omega_2\left(\frac{n}{N_2}\right)
V\left(\frac{\pi^{\frac{3}{2}}m n}{q^{\frac{3}{2}}}\right)\nonumber\\
&&e\left(\pm\frac{\overline{m}_1
\overline{q_2}n}{q_1}\pm\frac{\overline{m}_2
\overline{q_1}n}{q_2}\right)
+O\left(\frac{q^{\frac{5}{4}+\frac{3\theta}{2}+\varepsilon}}{\min\{q_1,q_2\}}\right).
\eea
We distinguish three cases.

Case I. For $M_2\leq q^{\beta_1}$, by partial integration once and Lemma 2.3, we have
\bna
&&\sum_{n\geq 1}\frac{\lambda_f(n)}{\sqrt{n}}
\omega_2\left(\frac{n}{N_2}\right)
V\left(\frac{\pi^{\frac{3}{2}}m n}{q^{\frac{3}{2}}}\right)
e\left(\pm\frac{\overline{m}_1
\overline{q_2}n}{q_1}\pm\frac{\overline{m}_2
\overline{q_1}n}{q_2}\right)\\
&=&-\int_0^{\infty}\left(\sum_{n\leq u}\lambda_f(n)
e\left(\pm\frac{\overline{m}_1
\overline{q_2}n}{q_1}\pm\frac{\overline{m}_2
\overline{q_1}n}{q_2}\right)\right)
\left(\frac{1}{\sqrt{u}}
\omega_2\left(\frac{u}{N_2}\right)
V\left(\frac{\pi^{\frac{3}{2}}m u}{q^{\frac{3}{2}}}\right)\right)'\mathrm{d}u\\
&\ll& N_2^{\varepsilon}.
\ena
Thus by (5.5),
\bea
S_{21}(\pm, M_2, N_2)\ll N_2^{\varepsilon}\sqrt{qM_2}+
\frac{q^{\frac{5}{4}+\frac{3\theta}{2}+\varepsilon}}{\min\{q_1,q_2\}}
\ll q^{\frac{1}{2}+\frac{\beta_1}{2}+\varepsilon}
+\frac{q^{\frac{5}{4}+\frac{3\theta}{2}+\varepsilon}}{\min\{q_1,q_2\}}.
\eea

Case II. For $M_2> q^{\beta_1}$ and $N_2\leq q^{\beta_2}$, we
apply Poisson summation formula to the $m$-sum to get
\bea
m\mbox{-sum}&=&
\sum_{(m,q)=1}\frac{1}{\sqrt{m}}
\omega_1\left(\frac{m}{M_2}\right)
V\left(\frac{\pi^{\frac{3}{2}}m n}{q^{\frac{3}{2}}}\right)e\left(\pm\frac{\overline{m}_1
\overline{q_2}n}{q_1}\pm\frac{\overline{m}_2
\overline{q_1}n}{q_2}\right)\nonumber\\
&=&
\sideset{}{^*}\sum_{\gamma_1 \bmod q_1}\,
\sideset{}{^*}\sum_{\gamma_2\bmod q_2}
e\left(\pm\frac{\overline{\gamma}_1
\overline{q_2}n}{q_1}\pm\frac{\overline{\gamma}_2
\overline{q_1}n}{q_2}\right)
\sum_{m\equiv \gamma_1 \bmod q_1
\atop m\equiv \gamma_2 \bmod q_2 }\frac{1}{\sqrt{m}}
\omega_1\left(\frac{m}{M_2}\right)
V\left(\frac{\pi^{\frac{3}{2}}m n}{q^{\frac{3}{2}}}\right)\nonumber\\
&=&
\sideset{}{^*}\sum_{\gamma_1 \bmod q_1}\,
\sideset{}{^*}\sum_{\gamma_2\bmod q_2}
e\left(\pm\frac{\overline{\gamma}_1
\overline{q_2}n}{q_1}\pm\frac{\overline{\gamma}_2
\overline{q_1}n}{q_2}\right)
\sum_{m\equiv \gamma_1 q_2\overline{q_2}+\gamma_2q_1\overline{q_1}\bmod q }\frac{1}{\sqrt{m}}
\omega_1\left(\frac{m}{M_2}\right)
V\left(\frac{\pi^{\frac{3}{2}}m n}{q^{\frac{3}{2}}}\right)\nonumber\\
&=&\frac{1}{q}\quad
\sideset{}{^*}\sum_{\gamma_1 \bmod q_1}\,
\sideset{}{^*}\sum_{\gamma_2\bmod q_2}
e\left(\pm\frac{\overline{\gamma}_1
\overline{q_2}n}{q_1}\pm\frac{\overline{\gamma}_2
\overline{q_1}n}{q_2}\right)
\sum_{m\in \mathbb{Z}}
e\left(\frac{(\gamma_1 q_2\overline{q_2}+\gamma_2q_1\overline{q_1}) m}{q}\right)
\nonumber\\
&&\int_{\mathbb{R}}\frac{1}{\sqrt{u}}
\omega_1\left(\frac{u}{M_2}\right)
V\left(\frac{\pi^{\frac{3}{2}}nu}{q^{\frac{3}{2}}}\right)
e\left(-\frac{m u}{q}\right)\mathrm{d}u\nonumber\\
&=&\frac{\sqrt{M_2}}{q}\sum_{m\in \mathbb{Z}}C(m,n;q)I(m,n),
\eea
where
\bna
I(m,n)=\int_{\mathbb{R}}\frac{\omega_1(u)}{\sqrt{u}}
V\left(\frac{\pi^{\frac{3}{2}}nM_2u}{q^{\frac{3}{2}}}\right)
e\left(-\frac{m M_2u}{q}\right)\mathrm{d}u\ll_j
\left(1+\frac{|m|M_2}{q}\right)^{-j}
\ena
and
\bna
C(m,n;q)&=&\sideset{}{^*}\sum_{\gamma_1 \bmod q_1}\,
\sideset{}{^*}\sum_{\gamma_2\bmod q_2}
e\left(\pm\frac{\overline{\gamma}_1
\overline{q_2}n}{q_1}\pm\frac{\overline{\gamma}_2
\overline{q_1}n}{q_2}\right)
e\left(\frac{(\gamma_1 q_2\overline{q_2}
+\gamma_2q_1\overline{q_1}) m}{q}\right)\\
&=&\sideset{}{^*}\sum_{\gamma_1 \bmod q_1}\,
e\left(\frac{\overline{q_2}m\gamma_1\pm\overline{q_2}n\overline{\gamma}_1
}{q_1}\right)
\sideset{}{^*}\sum_{\gamma_2 \bmod q_2}\,
e\left(\frac{\overline{q_1}m\gamma_2\pm\overline{q_1}n\overline{\gamma}_2
}{q_2}\right)\\
&\ll& (m,n,q_1)^{\frac{1}{2}}q_1^{\frac{1}{2}}
(m,n,q_2)^{\frac{1}{2}}q_2^{\frac{1}{2}}
\ll (m,n,q)^{1/2}q^{1/2}.
\ena
Plugging these estimates into (5.5), we obtain
\bea
S_{21}(\pm, M_2, N_2)\ll
&\ll&\frac{\phi(q)\sqrt{M_2}}{q\sqrt{q}}
\sum_{n\geq 1}\frac{|\lambda_f(n)|}{\sqrt{n}}
\omega_2\left(\frac{n}{N_2}\right)
\sum_{|m|<q^{1+\varepsilon}/M_2}(m,n,q)^{\frac{1}{2}}q^{\frac{1}{2}}
+\frac{q^{\frac{5}{4}+\frac{3\theta}{2}+\varepsilon}}{\min\{q_1,q_2\}}\nonumber\\
&\ll&\frac{q^{1+\varepsilon}\sqrt{N_2}}{\sqrt{M_2}}
+\frac{q^{\frac{5}{4}+\frac{3\theta}{2}+\varepsilon}}{\min\{q_1,q_2\}}
\ll q^{1-\frac{\beta_1}{2}+\frac{\beta_2}{2}+\varepsilon}
+\frac{q^{\frac{5}{4}+\frac{3\theta}{2}+\varepsilon}}{\min\{q_1,q_2\}}.
\eea

Case III. For $M_2> q^{\beta_1}$ and $N_2> q^{\beta_2}$, by (5.8), we
have
\bea
S_{21}(\pm, M_2, N_2)&=&\frac{\phi(q)\sqrt{M_2}}{q\sqrt{q}}
\sum_{|m|<\frac{q^{1+\varepsilon}}{M_2}}\,
\sideset{}{^*}\sum_{\gamma_1 \bmod q_1}\,
\sideset{}{^*}\sum_{\gamma_2\bmod q_2}
e\left(\frac{(\gamma_1 q_2\overline{q_2}
+\gamma_2q_1\overline{q_1}) m}{q}\right)\nonumber\\
&&\int_{\mathbb{R}}\frac{\omega_1(u)}{\sqrt{u}}e\left(-\frac{m M_2 u}{q}\right)
\sum_{n\geq 1}\frac{\lambda_f(n)}{\sqrt{n}}
\omega_2\left(\frac{n}{N_2}\right)
V\left(\frac{\pi^{\frac{3}{2}}n M_2 u}{q^{\frac{3}{2}}}\right)\nonumber\\
&&e\left(\pm\frac{(\overline{\gamma}_1
q_2\overline{q_2}+\overline{\gamma}_2
q_1\overline{q_1})n}{q_1q_2}\right)\mathrm{d}u+O\left(\frac{q^{\frac{5}{4}+\frac{3\theta}{2}+\varepsilon}}{\min\{q_1,q_2\}}\right)\nonumber\\
&=&\frac{\phi(q)\sqrt{M_2}}{q\sqrt{qN_2}}
\sum_{|m|<\frac{q^{1+\varepsilon}}{M_2}}\,
\sideset{}{^*}\sum_{\gamma\bmod q}
e\left(\frac{\gamma m}{q}\right)
\int_{\mathbb{R}}\frac{\omega_1(u)}{\sqrt{u}}
e\left(-\frac{m M_2 u}{q}\right)
\nonumber\\
&&\left(\sum_{n\geq 1}\lambda_f(n)e\left(\pm\frac{\overline{\gamma}n}{q}\right)\frac{\sqrt{N_2}}{\sqrt{n}}
\omega_2\left(\frac{n}{N_2}\right)
V\left(\frac{\pi^{\frac{3}{2}}n M_2 u}{q^{\frac{3}{2}}}\right)
\right)\mathrm{d}u\nonumber\\
&&+O\left(\frac{q^{\frac{5}{4}+\frac{3\theta}{2}+\varepsilon}}{\min\{q_1,q_2\}}\right).
\eea
Applying Voronoi formula in Lemma 2.4 to the $n$-sum we get
\bna
n\mbox{-sum}=q\sum_{n\geq 1}\frac{\lambda_f(n)}{n}
e\left(\pm \frac{\gamma n}{q}\right)\Psi_u^+\left(\frac{nN_2}{q^2}\right)
+q\sum_{n\geq 1}\frac{\lambda_f(n)}{n}
e\left(\mp \frac{\gamma n}{q}\right)\Psi_u^-\left(\frac{nN_2}{q^2}\right).
\ena
Correspondingly, $S_{21}(\pm, M_2, N_2)$ decomposes as two terms involving
$\Psi_u^+\left(\frac{nN_2}{q^2}\right)$ and $\Psi_u^-\left(\frac{nN_2}{q^2}\right)$,
respectively,
plus the $O$-term. Since the term involving $\Psi_u^-\left(\frac{nN_2}{q^2}\right)$
can be treated exactly the same as that involving $\Psi_u^+\left(\frac{nN_2}{q^2}\right)$,
we only consider the latter denoted by
\bna
S_{21}^{+}(\pm, M_2,N_2)
=\frac{\phi(q)\sqrt{M_2}}{\sqrt{qN_2}}\sum_{|m|<\frac{q^{1+\varepsilon}}{M_2}}
\sum_{n\geq 1}\frac{\lambda_f(n)}{n}H(m,n)
\sideset{}{^*}\sum_{\gamma\bmod q}
e\left(\frac{\gamma (m\pm n)}{q}\right),
\ena
where
\bna
H(m,n)&=&\int_{\mathbb{R}}\frac{\omega_1(u)}{\sqrt{u}}
e\left(-\frac{m M_2 u}{q}\right)\Phi_u^+\left(\frac{nN_2}{q^2}\right)
\mathrm{d}u.
\ena
Note that
\bna
\sideset{}{^*}\sum_{\gamma\bmod q}
e\left(\frac{\gamma (m\pm n)}{q}\right)&=&
\sideset{}{^*}\sum_{\gamma_1\bmod q_1}
e\left(\frac{\gamma_1 (m\pm n)}{q_1}\right)
\sideset{}{^*}\sum_{\gamma_2\bmod q_2}
e\left(\frac{\gamma_2 (m\pm n)}{q_2}\right)\\
&=&(q_1 1_{m\equiv \mp n \bmod q_1}-1)
(q_2 1_{m\equiv \mp n \bmod q_2}-1).
\ena
Thus
\bna
S_{21}^{+}(\pm, M_2,N_2)=R_1-R_2-R_3+R_4,
\ena
where
\bna
R_1&=&\frac{\phi(q)\sqrt{qM_2}}{\sqrt{N_2}}\sum_{|m|<\frac{q^{1+\varepsilon}}{M_2}}
\sum_{n\geq 1 \atop n\equiv \mp m \bmod q}\frac{\lambda_f(n)}{n}H(m,n),\\
R_2&=&\frac{q_1\phi(q)\sqrt{M_2}}{\sqrt{qN_2}}\sum_{|m|<\frac{q^{1+\varepsilon}}{M_2}}
\sum_{n\geq 1 \atop n\equiv \mp m \bmod q_1}\frac{\lambda_f(n)}{n}H(m,n),\\
R_3&=&\frac{q_2\phi(q)\sqrt{M_2}}{\sqrt{qN_2}}\sum_{|m|<\frac{q^{1+\varepsilon}}{M_2}}
\sum_{n\geq 1 \atop n\equiv \mp m \bmod q_2}\frac{\lambda_f(n)}{n}H(m,n),\\
R_4&=&\frac{\phi(q)\sqrt{M_2}}{\sqrt{qN_2}}\sum_{|m|<\frac{q^{1+\varepsilon}}{M_2}}
\sum_{n\geq 1}\frac{\lambda_f(n)}{n}H(m,n).
\ena
For $M_2> q^{\beta_1}$ and
$|m|<\frac{q^{1+\varepsilon}}{M_2}$, the condition
$m\equiv \mp n \bmod q$ with $m\neq \mp n$ implies that $n\asymp |k|q$ with $|k|\geq 1$. Thus
by (2.2),
\bna
R_1&\ll& q\frac{\sqrt{q M_2}}{\sqrt{N_2}}
\sum_{|m|<\frac{q^{1+\varepsilon}}{M_2}}
\frac{|\lambda_f(|m|)|}{|m|} \frac{N_2 |m|}{q^2}
+q\frac{\sqrt{q M_2}}{\sqrt{N_2}}
\sum_{|m|<\frac{q^{1+\varepsilon}}{M_2}}
\sum_{1\leq |k|\ll \frac{q^{1+\varepsilon}}{N_2}}
\frac{(q^2/N_2)^{\theta}}{|k|q}\\
&\ll&q^{\varepsilon}\frac{\sqrt{q N_2}}{\sqrt{M_2}}
+q^{\frac{3}{2}+\varepsilon}\frac{(q^2/N_2)^{\theta}}{\sqrt{M_2N_2}}\\
&\ll& q^{\frac{5}{4}-\beta_1+\varepsilon}
+q^{\frac{3}{2}-\frac{\beta_1}{2}-\frac{\beta_2}{2}+\varepsilon}
\left(q^{2-\beta_2}\right)^{\theta}.
\ena
Similarly,
\bna
R_2&\ll& \frac{\sqrt{q M_2}}{\sqrt{N_2}}
\sum_{|m|<\frac{q^{1+\varepsilon}}{M_2}}
\sum_{|n|<\frac{q^{1+\varepsilon}}{M_2}}
|\lambda_f(n)|\frac{N_2}{q^2}q_1
+\frac{\sqrt{q M_2}}{\sqrt{N_2}}
\sum_{|m|<\frac{q^{1+\varepsilon}}{M_2}}
\sum_{\frac{q^{1+\varepsilon}}{M_2}\leq|n|\leq
\frac{q^{2+\varepsilon}}{N_2} \atop n\asymp |k|q_1}
\frac{|\lambda_f(n)|}{n}q_1\\
&\ll&\frac{\sqrt{q M_2}}{\sqrt{N_2}}\frac{N_2}{q^2}q_1
\left(\frac{q^{1+\varepsilon}}{M_2}\right)^2
+\frac{\sqrt{q M_2}}{\sqrt{N_2}}
\frac{q^{1+\varepsilon}}{M_2}\left(\frac{q^2}{N_2}\right)^{\theta}\\
&\ll& q_1q^{\frac{5}{4}-2\beta_1+\varepsilon}
+q^{\frac{3}{2}-\frac{\beta_1}{2}-\frac{\beta_2}{2}+\varepsilon}
\left(q^{2-\beta_2}\right)^{\theta}
\ena
and
\bna
R_3\ll q_2q^{\frac{5}{4}-2\beta_1+\varepsilon}
+q^{\frac{3}{2}-\frac{\beta_1}{2}-\frac{\beta_2}{2}+\varepsilon}
\left(q^{2-\beta_2}\right)^{\theta}.
\ena
Finally,
\bna
R_4\ll q^{\frac{3}{2}-\frac{\beta_1}{2}-\frac{\beta_2}{2}+\varepsilon}.
\ena
We conclude that
\bea
S_{21}(\pm, M_2,N_2)\ll q^{\frac{5}{4}-\beta_1+\varepsilon}
+(q_1+q_2)q^{\frac{5}{4}-2\beta_1+\varepsilon}
+q^{\frac{3}{2}-\frac{\beta_1}{2}-\frac{\beta_2}{2}+\varepsilon}
\left(q^{2-\beta_2}\right)^{\theta}+
\frac{q^{\frac{5}{4}+\frac{3\theta}{2}+\varepsilon}}{\min\{q_1,q_2\}}.\nonumber\\
\eea
Taking
\bna
\beta_1=\frac{3}{4}+\frac{3\theta}{4(1+\theta)}, \qquad \beta_2=\frac{1}{2}+\frac{3\theta}{2(1+\theta)}.
\ena
Then Lemma 5.1 follows from (5.6), (5.8) and (5.10).
\end{proof}

$S_{22}$ and $S_{23}$ can be estimated similarly. We have
\begin{lemma}
For any $\varepsilon>0$, we have
\bna
S_{22}\ll q_1q^{-\frac{1}{8}+\frac{3\theta}{8(1+\theta)}+\varepsilon}
+q_1^{\frac{3}{2}}q^{-\frac{5}{8}+\frac{3\theta}{8(1+\theta)}+\varepsilon}
+q_1^3q^{-\frac{9}{8}-\frac{9\theta}{8(1+\theta)}+\varepsilon}
+\frac{q_1}{\min\{q_1,q_2\}}
q^{\frac{1}{4}+\frac{3\theta}{2}+\varepsilon}
+q^{\frac{3}{4}+\varepsilon}
\ena
and
\bna
S_{23}\ll q_2q^{-\frac{1}{8}+\frac{3\theta}{8(1+\theta)}+\varepsilon}
+q_2^{\frac{3}{2}}q^{-\frac{5}{8}+\frac{3\theta}{8(1+\theta)}+\varepsilon}
+q_2^3q^{-\frac{9}{8}-\frac{9\theta}{8(1+\theta)}+\varepsilon}
+\frac{q_2}{\min\{q_1,q_2\}}
q^{\frac{1}{4}+\frac{3\theta}{2}+\varepsilon}
+q^{\frac{3}{4}+\varepsilon}.
\ena
\end{lemma}

By (5.3), (5.4) and Lemmas 5.1 and 5.2, (3.3) follows.

\medskip

\medskip
\noindent
{\sc Acknowledgements.}
The author is supported by
the National Natural Science Foundation of China (Grant No. 11101239) and
Young Scholars Program of Shandong University, Weihai (Grant No. 2015WHWLJH04).

\end{document}